\documentclass[12pt]{amsart}
\usepackage{amsfonts}
\usepackage{graphicx}
\usepackage{color}
\usepackage{mathrsfs}
\usepackage{amssymb}

\newtheorem{thm}{Theorem}
\newtheorem{lemma}{Lemma}
\newtheorem{prop}{Proposition}

\newcommand{\insertimage}[2]{\includegraphics[scale=#1]{#2}}

\def\Z{\Bbb Z}

\def\N{\Bbb N}

\DeclareMathOperator\id{id}

\DeclareMathOperator\Ho{H} \DeclareMathOperator\Sp{Sp}
\DeclareMathOperator\PSp{PSp} 
\DeclareMathOperator\image{Im} \DeclareMathOperator\MCG{MCG}

\DeclareMathOperator\prob{prob}

\newcommand{\Rmnum}[1]{\expandafter\@slowromancap\romannumeral #1@}
\makeatother

\begin{document}
\begin{abstract} We prove that the set of non-pseudo-Anosov elements
in the Torelli group is exponentially small. This answers a question
of Kowalski \cite{Ko}.
\end{abstract}

\title{Sieve methods in group theory \Rmnum{2}: \\ The Mapping Class Group}
%\subjclass[2000]{Primary 54A25; Secondary 03E04.}


\author{Alexander Lubotzky and Chen Meiri}
\address{Einstein Institute of Mathematics \\Hebrew University\\Jerusalem 90914, Israel}
\email{alexlub@math.huji.ac.il, chen.meiri@mail.huji.ac.il}
\maketitle
\date{\today}

\section{Introduction}

Let $S$ be an orientable closed surface of genus $g \ge 1$ and
denote its mapping class group by $\MCG(S)$. Thurston divided the
elements of the $\MCG(S)$ into three types: pseudo-Anosov (PA for
short), reducible and periodic (see Theorem \ref{NTC} below). He
conjectured that `most' of the elements of $\MCG(S)$ are PA. This
was proved by Maher \cite{Ma}. A stronger result was shown by Rivin
\cite{Ri} (and reproved by Kowalski \cite{Ko} in  a more conceptual
form). They proved that the set of non-PA elements is exponentially
small. Let us define this notion:

Let $\Gamma$ be a finitely generated group and fix a symmetric
generating set $\Sigma$ of $\Gamma$ which satisfies an odd relation.
Random walks on the Cayley graph $\textrm{Cay}(\Gamma,\Sigma)$ can
be used to `measure' subsets $Z\subseteq \Gamma$ by estimating the
probability, $\prob_{\Sigma,k}(Z)$, that the $k^{\text{th}}$-step of
a random walk belongs to $Z$ for larger and larger values of $k$'s.
We say that $Z$ is \emph{exponentially small with respect to
$\Sigma$} if there exist constants $c,\alpha>0$ such that
$\prob_{\Sigma,k}(Z) \le ce^{-\alpha k}$ for all $k \in \N$. The set
$Z$ is called \emph{exponentially small} if it is exponentially
small with respect to all symmetric generating sets $\Sigma$ which
satisfies an odd relation (so the Cayley graph is not bi-partite
graph).

Rivin's result states:
\begin{thm}\label{MCG thm}{$($\cite{Ri}, see also \cite{Ko}$)$}. The set of non-PA
elements of $\MCG(S)$ is exponentially small.
\end{thm}

The proof of this result uses the epimorphism
$\pi:\MCG(S)\rightarrow \Sp(2g,\Z)$ (which is induced by the action
of $\MCG(S)$ on the homology of $S$). It is shown in \cite{BC} that
if $\gamma \in \MCG(S)$ is not PA then the characteristic polynomial
$f(x)$ of $\pi(\gamma)$ satisfies one of the following
possibilities:
\begin{itemize}
\item[a.] $f(x)$ is reducible in $\mathbb{Q}[x]$.
\item[b.] $f(x)$ has a root which is a root of unity.
\item[c.] There is $d \ge 2$ and polynomial $g(x)$ such that $f(x)=g(x^d)$
\end{itemize}

Thus, it is enough to show that the set of elements of $\Sp(2g,\Z)$
which satisfy at least one of the above three conditions is
exponentially small. This is not a hard task for conditions ($b$)
and ($c$) so we only focus on condition $(a)$. The proof that the
set of elements which satisfy condition $(a)$ is exponentially small
uses the `large sieve method' (implicitly in \cite{Ri} and
explicitly in \cite{Ko}). For our purpose, this method can be
summarized in the following theorem  which follows from Theorem B of
\cite{LM}

\begin{thm}\label{sieve theorem1} Let $\Gamma$ be a finitely
generated group and let $\mathcal{P}$ be a set of all but finitely
many primes . Let $(N_p)_{p \in \mathcal{P}}$ be a series of finite
index normal subgroups of $\Gamma$. Assume that there is a constant
$d \in \N^+$ such that:
\begin{itemize}
\item[1.] $\Gamma$ has property-$\tau$ with respect to the series of
normal subgroup $(N_p \cap N_q)_{p,q\in \mathcal{P}}$.
\item[2.] $|\Gamma_p| \le p^d$ for every $p \in \mathcal{P}$ where $\Gamma_p:=\Gamma/N_p$.
\item[3.] The natural map $\Gamma_{p,q}\rightarrow \Gamma_p \times
\Gamma_q$ is an isomorphism   for every distinct $p,q\in
\mathcal{P}$ where $\Gamma_{p,q}:=\Gamma /(N_p \cap N_q$).
\end{itemize} Then a subset $Z\subseteq \Gamma$ is exponentially small if there is $c>0$ such that:
\begin{itemize}
\item[4.] $\frac{|Z_p|}{|\Gamma_p|}\le 1-c$ for every $p \in \mathcal{P}$ where $Z_p:=ZN_p/N_p$.
\end{itemize}
\end{thm}

Theorem \ref{sieve theorem1} can be used to show that the set  $Z
\subseteq \Sp(2g,\Z)$ consists of matrices with reducible
characteristic polynomials is exponentially small. Indeed, the
symplectic group $\Sp(2g,\Z)$ has property-$\tau$ with respect to
the family $\{N_q \mid q \text{ is a square free positive
integer}\}$ where $N_q$ is the kernel of the modulo-$q$ homomorphism
$\Sp(2g,\Z)\rightarrow \PSp(2g,\Z/qZ))$ (in fact, the group
$\Sp(2g,\Z)$ even has property (T) for $g \ge 2$). Condition 2 holds
for $d=4g^2$ and condition 3 is well known. Condition 4 is due to
Chavdarov \cite{Ch} who proved that there is a constant $c\in(0,1)$
such that for every prime $p$ the proportion of the elements in
$\Sp(2g,\mathbb{Z}/p\mathbb{Z})$ with reducible characteristic
polynomial is at most $c$.

The proof sketched above gives no information on the Torelli
subgroup $\mathcal{T}(S):=\ker \pi$. The group $\mathcal{T}(S)$ is
trivial for $g=1$ and infinitely generated for $g=2$. On the other
hand, $\mathcal{T}(S)$ is finitely generated group for $g \ge 3$ so
from now on we assume $g \ge 3$. In the early days it was not even
known if $\mathcal{T}(S)$ contains PA elements. However, it does and
Maher result even showed that `most' of the elements in a ball of
radius $n$ are $PA$. Kowalski [Ko, page 135] asked if the stronger
result is valid also for $\mathcal{T}(S)$. In this paper we prove
that this is indeed the case:
\begin{thm}\label{torelli thm} The set of non-PA
elements of $\mathcal{T}(S)$ is exponentially small.
\end{thm}
Our proof also uses the sieve method but instead of $\pi$ we look at
the actions of $\mathcal{T}(S)$ on the homologies of the $2^{2g}-1$
2-sheeted covers of $S$ (called prym representations in \cite{Loo},
see also \cite{GL} and \cite{GL2}). These covers give $2^{2g}-1$
different homomorphisms $\pi_1,\ldots,\pi_{2^{2g}-1}$ of
$\mathcal{T}(S)$ into $\Sp(2g-2,\Z)$. We show that if $\gamma \in
\mathcal{T}(S)$ is not PA then there is $1 \le i \le 2^{2g}-1$ such
that the characteristic polynomial of $\pi_i(\gamma)$  is reducible.
We also show (following \cite{Loo} and \cite{GL2}) that
$\pi_i(\mathcal{T}(S))$ is a finite index subgroup of $\Sp(2g-2,\Z)$
and hence also has property-$\tau$ with respect to congruence
subgroups. We can therefore use the sieve method in a similar manner
to the use for the mapping class group case.

After this work was announced in \cite{Lu2} and a  draft was written
we learnt that Justin Malestein and Juan Souto also proved Theorem
\ref{main thm} \cite{MS}. The main idea is similar in both proofs,
but some details are done differently.

\section{The Torelli group.}

\subsection{The mapping class group and the Torelli subgroup}

Let $S$ be an orientable connected compact surface of genus $g \ge
2$, $S$ is homeomorphic to a connected sum of g tori. For $1\le k\le
g-1$, let $c_k$ be the curve separating the left $k$ tori for the
right $g-k$ ones, as shown in Figure \ref{A torus of genus $g$.}.
%There is a unique (up to diffeomorphism) smooth structure on $\Sigma_g$.
%The fundamental group of $\Sigma_g$, denoted by $\pi_1(\Sigma_g)$,
%is isomorphic to $$\langle
%a_1,b_1,\ldots,a_g,b_g \mid [a_1,b_1]\cdots [a_g,b_g]=1\rangle.$$ %picture
%The \emph{mapping class group} $\Gamma_S$ of $S$ is the group of
%isotopy class of orientation-preserving homeomorphisms of $S$.
%$\Gamma_S$ acts on
The first homology group $\Ho_1(S,\mathbb{Z})$  is isomorphic to
$\Z^{2g}$. If $c$ is a closed curve of $S$ then $\bar{c}$ denotes
the homology class of it. The elements $\bar{a}_1,\bar{b}_1,\ldots
,\bar{a}_g,\bar{b}_g$ form a basis of $\Ho_1(S,\mathbb{Z})$ where
$a_1,b_1,\ldots,a_g,b_g$ are the curves shown in Figure \ref{A torus
of genus $g$.}.

If $d_1,d_2$ are simple closed curves on $S$ with finite
intersection then  their \emph{intersection number} $i({d}_1,{d}_2)$
is the sum of the indices of the intersection points of $d_1$ and
$d_2$, where an intersection point is of index $+1$ when the
orientation of the intersection agrees with the orientation of S,
and $-1$ otherwise. The intersection number induces a symplectic
form $(\cdot,\cdot)_S$ on $\Ho_1(S,\mathbb{Z})$ in the following
way: Given two homology classes we choose representatives $d_1,d_2$
with finite intersection and define
$(\bar{d}_1,\bar{d}_2)_S:=i({d}_1,{d}_2)$. In particular,
$(\bar{a}_i,\bar{b}_i)_S=\delta_{ij}$, $(\bar{a}_i,\bar{a}_j)_S=0$
and $(\bar{b}_i,\bar{b}_j)_S=0$ for all the integers $1 \le i,j\le
g$, i.e. $\bar{a}_1,\bar{b}_1,\ldots
,\bar{a}_g,\bar{b}_g$ is a sympletic basis. %We identify the
%symplctic group $\Sp(Ho_1(S,\mathbb{Z})$ with $\Sp(2g,\mathbb{Z})$
%via this basis.

The \emph{mapping class group} $\MCG(S)$ of $S$ is the group of
isotopy classes of orientation-preserving homeomorphisms of $S$. We
will denote an isotopy classes with representative $\psi$ by
$[\psi]$. The group $\MCG(S)$ acts on $\Ho_1(S,\mathbb{Z})$ and
preserves its symplectic form. Hence, this action induces a
homomorphism
$$\gamma:\MCG(S)\rightarrow \Sp (\Ho_1(S,\mathbb{Z})).$$
$\mathcal{T}(S):=\ker (\gamma)$ is called the \emph{Torelli group}
of $S$.

While the action of $\MCG(S)$ on $\Ho_1(S,\mathbb{Z})$ can be used
to investigate  its elements, it does not tell us much about the
Torelli group elements. To overcome this problem we will investigate
the action of the Torelli group on the homology of the double covers
of $S$.
\begin{figure}[ht]
\begin{center}
\insertimage{.7}{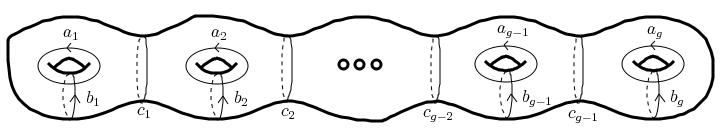} \caption{A surface of genus $g$.}
\label{A torus of genus $g$.}
\end{center}
\end{figure}

\subsection{Pseudo-Anosov elements}

%We start this section with the definition of a Pseudo-Anosov
%element. However, we will not use this definition anywhere in this
%note. Instead we will use the below Nielsen-Thurston classification.
 %Let $\varphi \in \Hom^+(\sigma_g)$. $\varphi$ acts on $\H_1(\Sigma_g)$.
 %Assume that there is a non trivial circle in $\Sigma_g$ which is invariant under
 %$\varphi$. Then there is a non trivial curve $C$ such that
 %$\varphi(C)=C$ or $\varphi(C)=C^{-1}$. By replacing  $\varphi$ with a suitable conjugate of it
 %we can assume that $C=a_1$ or $C=[a_1,b_1]\cdots[a_k,b_k]$ for some
 %$1 \le k \le g-2$ and that $\varphi(C)=C$. Moreover, by furthere replacing
 %$\varphi$ with $\varphi^2$ we can assume that if
 %$C=[a_1,b_1]\cdots[a_k,b_k]$ then $\varphi$ leave invariant each of
 %the component of $\Sigma_g\setminus C$.
%\begin{dfn}\label{pseudo-anosov} An element $\Psi \in \Gamma_g$ is called  \emph{pseudo-Anosov} if
%there is a pair of transverse measured foliations $(F^u, \mu_u)$ and
%$(F^s, \mu_s)$ on $S$, a number $\lambda > 1$, and a   $\psi\in
%\Psi$ so that  $\psi \cdot(F^u, \mu_u)=(F^u, \lambda\mu_u)$ and
%$\psi\cdot(F^s, \mu_s)=(F^s, \lambda^{-1}\mu_s)$.  \end{dfn}

We start this section with the Nielsen-Thurston classification of
the elements of the mapping class group . An element $\Psi \in
\Gamma_g$ is called  \emph{pseudo-Anosov} (PA for short) if there is
a pair of transverse measured foliations $(F^u, \mu_u)$ and $(F^s,
\mu_s)$ on $S$, a number $\lambda > 1$, and a $\psi\in \Psi$ so that
$\psi \cdot(F^u, \mu_u)=(F^u, \lambda\mu_u)$ and $\psi\cdot(F^s,
\mu_s)=(F^s, \lambda^{-1}\mu_s)$ (see \cite{a primer}). For our
purpose, we do not need the details of this definition and we can
use the following theorem:
\begin{thm}[see \cite{a primer}, Theorem 12.1]\label{NTC} An element
$\Psi \in \MCG (S)$ is pseudo-Anosov if
none of the following two possibilities holds:%\footnote{The usual
%definition of pseudo-Anosov elements is rater complicated so we use
%another definition which was proved to be equivalent by Thurston.}
\begin{enumerate}
\item[1.] $\Psi^n=[\id]$ for some $n\in\mathbb{N}^+$.
\item[2.] There is a finite non empty disjoint collection of non homotopic non trivial circles $C_i$
in $S$ and a $\psi\in\Psi$ such that $\psi(\mathcal{C}) =
\mathcal{C}$, where $\mathcal{C}$ is the union of the circles $C_i$
($C_i$ is non trivial if is not contractible to a point).
\end{enumerate}
We call $\Psi$ periodic in case $(1)$ and reducible in case $(2)$.
\end{thm}
It is known that the  Torelli group is torsion free. Hence, in order
to show that an elements $\Psi\in \mathcal{T}(S)$ is pseudo-Anosov
we will have to show that it is not reducible. In fact, we will use
a result of Ivanov ( \cite{Subgroups of Teichm¨uller}, Corollary
$1.8$) which implies:
\begin{prop} Let $\Psi\in \mathcal{T}(S)$. If $\Psi$ is not pseudo-Anosov then there
is a non-trivial circle $C$ and $\psi \in \Psi$ such that
$\Psi(C)=C$.
\end{prop}

%$\gamma(\varphi)$ acts on $H_1(\Sigma_g)$ and by the above
%definition it can be shown that if the characteristic polynomial of
%$\gamma(\varphi)$ is irreducible and non of its roots is a root of
%unity then $\varphi$ is pseudo-Anosov. We can't use this criterion
%since the Torelli group act trivially on $H_1(\Sigma_g)$. Therefore,
%we will extend this criterion to the action on the homology group of
%a suitable cover.

A circle $C$ on $S$ is \emph{non-separating} if $S \setminus C$ is
connected. If $C$ and $C'$ are non separating circles then there is
a homeomorphism which sends $C$ to $C'$. Hence, if $C$ is a
non-separating circle then there is a homeomorphism which sends $C$
to $a_1$.

A circle $C$ of $S$ is \emph{separating} if $S \setminus C$ is not
connected. If $C$ is a non trivial separating circle then it follows
from the classification theorem of compact surfaces that there is $1
\le k \le g-2$ and a homeomorphism which sends $C$ to the curve
$c_k$ shown in Figure 1 (in fact, by symmetry we can assume that $k
\le \frac{g-1}{2}$).

Summing all this we get the following:
\begin{lemma}\label{simple criterion} Let $\Psi \in \mathcal{T}(S)$. If $\Psi $ is not
psuedo-Anosov then there is $\psi\in\Psi$ and an orientation
preserving homeomorphism $\nu$ of $S$ such that
one of the following holds:%\footnote{$\cdot$ denotes the action $\Hom^+(\sigma_g)$ induces on $\pi_1(\Sigma_g))$.}
\begin{enumerate}
\item[1.] $\psi(\nu(a_1)) = \nu(a_1)$ or $\psi(\nu(a_1)) = {\nu(a_1)}^{-1}$.
\item[2.] $\psi(\nu(c_k)) = {\nu(c_k)}$ or $\psi({\nu(c_k)}) ={\nu(c_k)}^{-1}$  for some $1\le k\le
g-2$.
\end{enumerate}
\end{lemma}

\subsection{The action of a psuedo-Anosov element on the homology of
a double cover.}

Let $M$ be the bordered surface obtained from $S$ after cutting
along the curve $b_g$ (Figure 2). So, $S$ is a quotient space of $M$
under the identification of $\dot{b}_g$ and $\hat{b}_g$.  We take
two copies  $M^\diamond$,$M^*$ of $M$  and glue them together by
identifying $\dot{b}_g^*$ with $\hat{b}_g^\diamond$  and
$\hat{b}_g^*$ with $\dot{b}_g^\diamond$ (Figure 3). The resulting
surfaces $\tilde{S}$ together with the natural  map
$\varphi:\tilde{S}\rightarrow S$  is a double cover of $S$ of genus
$2g-1$.
\begin{figure}[ht]
\begin{center}
\insertimage{.7}{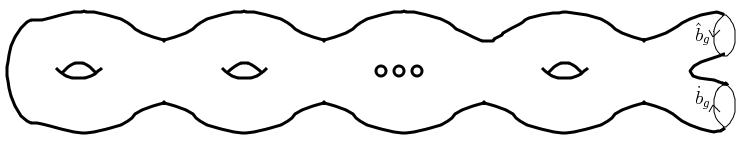} \caption{The surface after cutting along
$b_g$.} \label{Fig2} \insertimage{.7}{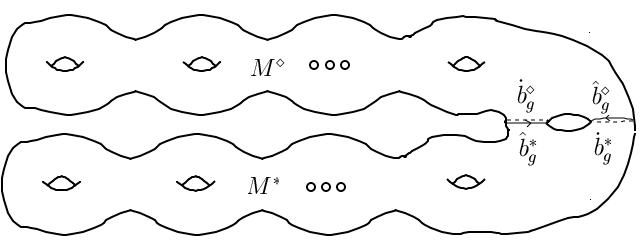} \caption{The double
cover.} \label{Fig3}
\end{center}
\end{figure}

% and $\Hom_1(\tilde{S},\mathbb{Z})$ is isomorphic to
%$\mathbb{Z}^{2g-1}$.
%We want to define an explicit basis for
%$\Hom_1(\tilde{S},\mathbb{Z})$

%The automorphism group of $(\tilde{S},\varphi)$ is cyclic of order
%two. If $\sigma$ is the generator of it then $\sigma$ acts on
%$\Ho_1(\tilde{S},\mathbb{Z})$. Let
%$$\mathcal{H}:=\{h\in \Ho_1(\tilde{S},\mathbb{Z}\mid\sigma(h)=-h\}$$
%$\mathcal{H}$ is a subgroup $\Ho_1(\tilde{S},\mathbb{Z})$ which is
%invariant under the action of $\MCG(\tilde{S})$.
Let $\mathcal{H}$ be the kernel of the homomorphism
$\Ho_1(\tilde{S},\mathbb{Z})\rightarrow \Ho_1({S},\mathbb{Z})$
induced by $\varphi$. The group $\mathcal{H}$ is isomorphic to
$\mathbb{Z}^{2g-2}$ with $z_1,z_{-1},\ldots,z_{g-1},z_{1-g}$ as a
free basis where $z_i:=\bar{a}_i^*-{\bar{a}_i^\diamond}$ and
$z_{-i}:=\bar{b}_i^*-{\bar{b}_i^\diamond}$ for $1\le i \le g-1$.

The symplectic form $(\cdot,\cdot)_{\tilde{S}}$ of
$\Ho_1(\tilde{S},\mathbb{Z})$ induces a related symplectic form
$(\cdot,\cdot)_{\mathcal{H}}:=\frac{1}{2}(\cdot,\cdot)_{\tilde{S}}$
on $\mathcal{H}$. The form $(\cdot,\cdot)_{\mathcal{H}}$ satisfies:
$$(z_i,z_j)_{\mathcal{H}}:= \left\{
\begin{array}{ccc}
1  & \textrm{if}& i=-j \textrm{ and } 0<i \\
-1  & \textrm{if}&i=-j\textrm{ and } i<0 \\
0 & \textrm{otherwise}&
\end{array} \right.$$

Note that the identity map on $S$ has two lifts to $\tilde{S}$. One
is the identity and the other one is $\xi$ which for every $q \in S$
substitutes the two lifts of $q$.

Choose a point $p\in S$ and a lift $\tilde{p}$ of it to $\tilde{S}$.
For every $\Psi \in \mathcal{T}(S)$ we choose a representative
$\psi\in\Psi$ such that $\psi(p)=p$. Since $(\tilde{S},\varphi)$ is
an abelian cover of $S$ and $\mathcal{T}(S)$ acts trivially on
$\Ho_1(S,\mathbb{Z})$ we can lift $\psi$ to $\tilde{S}$ in two ways,
one which fixes $\tilde{p}$ and one which switches $\tilde{p}$ with
the other lift of $p$. Let $\tilde{\psi}$ be the first lift (the
second one is $\xi \circ \tilde{\psi}$). This defines a map
$\lambda: \mathcal{T}(S) \rightarrow \MCG(\tilde{S})$ which sends
every $\Psi \in \mathcal{T}(S)$ to the isotopic class of
$\tilde{\psi}$. If $\Psi \in\mathcal{T}(S)$ then $\lambda(\Psi)$
acts on $\Ho_1(\tilde{S},\mathbb{Z})$ and leaves $\mathcal{H}$
invariant. It also preserves $(\cdot,\cdot)_{\tilde{S}}$ and
$(\cdot,\cdot)_{\mathcal{H}}$. This define a map from
$\mathcal{T}(S)$ to $\Sp(\mathcal{H})$. However, this map may depend
on the choice of representative $\psi$ of $\Psi$ and may fail to be
a homomorphism. But let us now compose $\lambda$ with $\pi$, the
homomorphism from isotopy classes of $\MCG(\tilde{S})$ which leaves
$\mathcal{H}$ invariant and preserves $(\cdot,\cdot)_{\mathcal{H}}$
into $\PSp(\mathcal{H})$. We define:
$$\rho:=\pi\circ\lambda:\mathcal{T}(S)\rightarrow
\PSp(\mathcal{H}).$$
\begin{prop} The map $\rho$ is a homomorphism.
\end{prop}
\begin{proof} Since $\pi$ is a homomorphism is is suffices to show
that $\rho(\Psi)$ does not depend on the representative $\psi$ which
has been chosen for $\Psi$.

Let $\Psi \in \mathcal{T}(S)$ and let $\psi_0,\psi_1 \in \Psi$. Let
$\tilde{\psi}_0$ and $\tilde{\psi}_1$ be the lifts of $\psi_0$ and
$\psi_1$ which preserve the point $\tilde{p}$. There is an isotopy
$(\psi_t)_{t \in [0,1]}$ between $\psi_0$ and $\psi_1$. This isotopy
can be lifted to isotopy $(\hat{\psi}_t)_{t \in [0,1]}$ in
$\tilde{S}$ such that $\hat{\psi}_0=\tilde{\psi}_0$. Note that
$\hat{\psi}_1$ is a lift of $\psi_1$ but it does not have to
preserve $\tilde{p}$. If $\hat{\psi}_1(\tilde{p}) = \tilde{p}$ then
$\hat{\psi}_1=\tilde{\psi}_1$ which means that $\tilde{\psi}_0$ and
$\tilde{\psi}_1$ are isotopic so $\pi([\psi_0])=\pi([\psi_1])$. On
the other hand, if $\hat{\psi}_1(\tilde{p}) \ne \tilde{p}$ then
$\hat{\psi}_1 = \xi \circ \tilde{\psi}_1$ when $\xi$ is the
non-identity lift to $\tilde{S}$ of the identity on $S$ and
$\pi([\psi_0])=\pi([\xi
\circ\psi_1])=\pi([\xi])\pi([\psi_1])=\pi([\psi_1])$.
\end{proof}

The goal of the next proposition is to show that the image of $\rho$
is of finite index in $\PSp(2g-2,\Z)$. The proof of this proposition
is by analyzing certain \emph{Dehn twists} so we briefly recall
their definition (for more details see \cite{a primer}, $2.2$ and
$7.1$).

Let $t$ be a simple closed curve of an oriented surface $S$ and let
$N$ be a regular neighborhood of it homoeomorphic to an annulus (the
homoeomorphism is required to preserve orientation), which we
consider parameterized as
$$\{(r,\theta)\mid 1\le r\le 2\ \ \wedge \ 0 \le \theta < 2\pi \}$$
where $t$ is identified with $\{(\frac{3}{2},\theta) \mid 0 \le
\theta < 2\pi)\}$.

A \emph{Dehn twist} $D_t$ along $t$ is the homeomorphism given by
identity outside $N$ and by the map $(r,\theta)\mapsto(r,\theta+2\pi
r)$ on $N$. The isotopic class of this homeomorphism does not depend
on the neighborhood $N$. The action of $D_t$ induced on
$\Ho_1(S,\mathbb{Z})$ is given by $h\mapsto h+(h,\bar{t})_S\bar{t}$.
In particular if $\bar{t}=0$ then $[D_t]\in \mathcal{T}(S)$.

%By choosing an appropriate neighborhoods for $a_1,\cdots,a_{g-1}$
%and $b_1,\cdots,a_{b-1}$ and the point $p$ to lie outside of this
%neighborhoods we can repeat the above arguments and extend $\rho$ to
%a homomorphism $\bar{\rho}: \rightarrow \PSp(2g-2,\Z)$ where

The following Proposition is a consequence of the more general
results in \cite{Loo}. For completeness we give the proof of our
concrete case.

\begin{prop}\label{rho} The image of $\rho$ is a finite index subgroup of
$\PSp(\mathcal{H})$.
\end{prop}
\begin{proof}
For $i,j\in \{\pm 1,\ldots,\pm (g-1)\}$ with $i\ne \pm j$ we define:
$$
\begin{array}{ccc}
T_i :\mathcal{H}\rightarrow\mathcal{H} & &T_{i,j}
:\mathcal{H}\rightarrow\mathcal{H} \\
h\mapsto h+(h,{z}_{-i})_{\mathcal{H}}{z}_{-i}& & h\mapsto
h+(h,{z}_{-i})_{\mathcal{H}}{z}_{-j}+(h,{z}_{-j})_{\mathcal{H}}{z}_{-i}
\end{array} $$
The  $T_i$'s and $T_{i,j}$'s  are called \emph{elementary symplectic
transformations} of $\mathcal{H} $ with respect to the basis
$z_1,z_{-1},\ldots,z_{g-1},z_{1-g}$ and they generate
$\Sp(\mathcal{H})$ \cite{HO}. Let $E$ be the set of elementary
symplectic transformations  with respect to the above basis. Tits
proved that for every $k\in \mathbb{N}^+$ the set $E^k:=\{T^k\mid
T\in E\}$ generates a finite index subgroup of
$\Sp(2g-2,\mathbb{Z})$ (\cite{Tit}). Let $\bar{T}_i\in
\PSp(2g-2,\Z)$ and $\bar{T}_{i.j}\in\PSp(2g-2,\Z)$ be the images of
$T_i$ and $T_{i,j}$. In order to prove the Proposition it suffices
to show that $\textrm{Im}\rho$ contains $\bar{E}^4$ where
$\bar{E}^4:=\{\bar{T} \mid T\in E^4\}$.

Fix $1 \le i \le g-1$. We start by showing that
$\bar{T}_i^4\in\textrm{Im}\rho$. Let $t$ be the simple closed  path
of $S$ drawn in Figure 4. Then $[D_t]\in \mathcal{T}(S)$ since
$\bar{t}=0$. The curve  $t$ has two disjoint lifts ${t}_1$,${t_2}$
to $\tilde{S}$ as shown in figure 5. Thus, $\rho
({[D_t]})=\pi([D_{t_1}])\circ \pi([D_{t_2}])$ (since the two lifts
are disjoint, the isotopic classes of their Dehn twists commute, so
the order of the multiplication is not important). Since
$(\cdot,\cdot)_\mathcal{H}=\frac{1}{2}(\cdot,\cdot)_{\tilde{s}}$ and
$\bar{t}_1=-\bar{t}_{2}={z}_{-i}$ then
$\pi[D_{t_1}]=\pi[D_{t_2}]=\bar{T}_i^2$ and
$\rho([D_t])=\bar{T}^4_{i}$.
\begin{figure}
\begin{center}
\insertimage{.5}{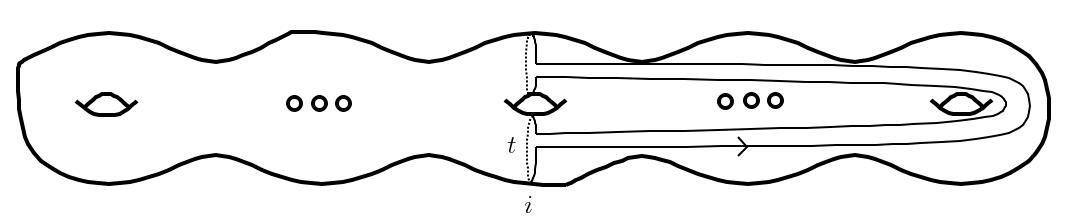}\caption{} \end{center}
\end{figure}
\begin{figure}
\begin{center}
\insertimage{.6}{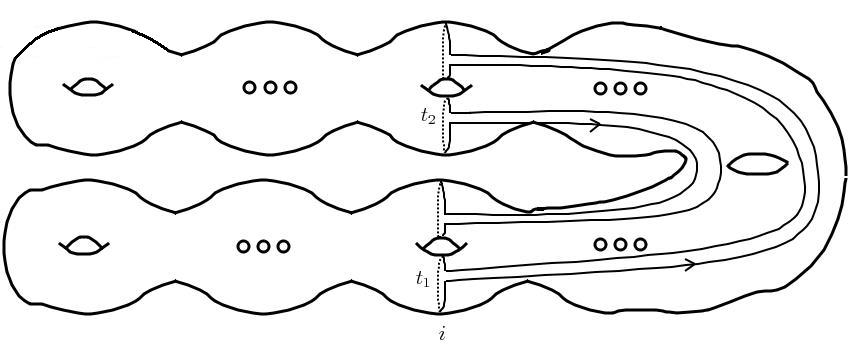}\caption{}
\end{center}
\end{figure}
\begin{figure}
\begin{center}
\insertimage{.55}{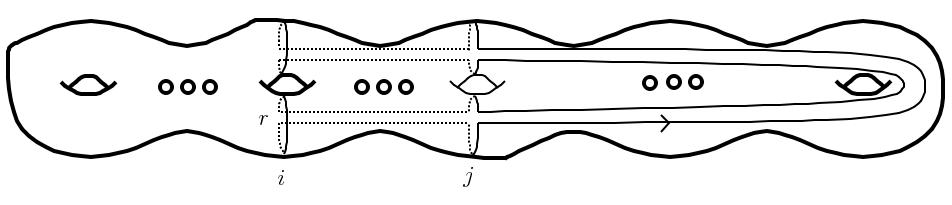}\caption{}
\end{center}
\end{figure}
\begin{figure}
\begin{center}
\insertimage{.6}{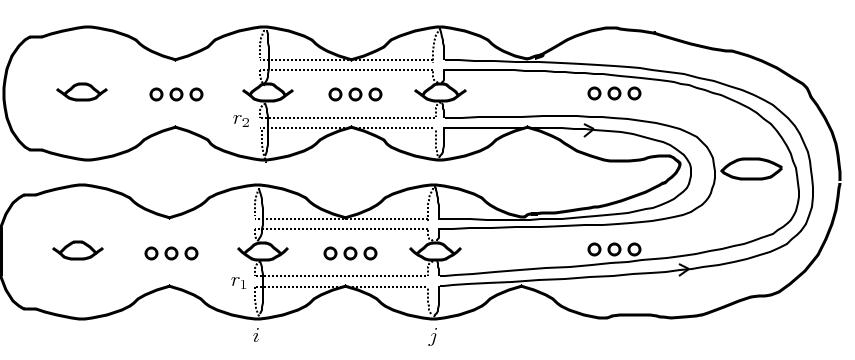}\caption{}
\end{center}
\end{figure}

Next, we show that $\bar{T}^4_{i,j}\in\textrm{Im}\rho$ where $1 \le
i\ne j \le g-1$. Let $r$ be the closed simple path of $S$ drawn in
Figure 6. Then $[D_r]\in \mathcal{T}(S)$ since $\bar{r}=0$. The
curve $r$ has two disjoint lifts ${r}_1$,${r_2}$ to $\tilde{S}$ as
shown in figure 7. Thus,
$\rho([D_r])=\pi([D_{r_1}])\circ\pi([D_{r_2}])$ and
$\bar{r}_1=-\bar{r}_2=\bar{z}_{-j}-\bar{z}_{-i}$. Therefore,
%$\lambda([D_r])$ acts on $\mathcal{H}$ in the following way:
%$$h \mapsto  h+4(h,z_{-j})_{\mathcal{H}}z_{-j}+4(h,z_{-i})_{\mathcal{H}}z_{-i}-4(h,z_{-j})_{\mathcal{H}}z_{-i}-4(h,z_{-i})_{\mathcal{H}}z_{-j}.$$
%In other words,
$\rho([D_r])=\bar{T}^4_{i}\circ \bar{T}^4_{j}\circ
\bar{T}^{-4}_{i,j}$. Thus, also $\bar{T}^{4}_{i,j}\in \image \rho$
since $\bar{T}^4_{i},\bar{T}^4_{j}\in \image \rho$.

This handle the case of positive $i$ and $j$. For the general case
we can change coordinates and argue as follows:

Fix $\Psi_1,\Psi_2\in \mathcal{T}(S)$ such that
$\rho(\Psi_1)=\bar{T}_i^4$ and $\rho({\Psi_2})=\bar{T}_{i,j}^4$ for
$1 \le i \ne j \le g-1$. Let $\Delta_i\in \MCG(S)$ be the isotopic
class of the composition $D_{b_i}\circ D_{a_i}\circ D_{b_i}$ of Dehn
twists.  By the formula given in the previous paragraph for the
action of a Dehn twist on the first homology we see that:
$$\label{finite index lemma delta}\begin{array}{cccc} \Delta_i
(\bar{a}_i) = \bar{b}_i, & \Delta_i(\bar{b}_i) = -\bar{a}_i, &
\Delta_i(\bar{a}_j) = \bar{a}_j ,&
 \Delta_i(\bar{b}_j) = \bar{b}_j.
\end{array} $$ Hence:
$$\begin{array}{ccc}
\rho(\Delta_i\Psi_1\Delta_i^{-1})=\bar{T}_{-i}^4 & \ \ &
\rho(\Delta_i^{-1}\Psi_2\Delta_i)=\bar{T}_{-i,j}^4 \\
\rho(\Delta_j^{-1}\Psi_2\Delta_j)=\bar{T}_{i,-j}^4 & \ \ &
\rho(\Delta_i^{-1}\Delta_j^{-1}\Psi_2\Delta_j\Delta_i)=\bar{T}_{-i,-j}^4
\end{array}. $$
Thus, if $\bar{T}_i^4,\bar{T}^4_{i,j}\in\textrm{Im}\rho$ then so are
$\bar{T}_{\pm i}^4,\bar{T}^4_{\pm i,\pm j}\in\textrm{Im}\rho$.
\end{proof}

An element of $\PSp(2g-2,\Z)$ has two lifts to an elements of
$\Sp(2g-2,\Z)$. However, the characteristic polynomials of these
lifts are all reducible or all irreducible. Thus, we can talk about
the reducibility of an element of $\PSp(2g-2,\Z)$. We are ready to
prove the main proposition of this section.

\begin{prop}\label{main prop} There is a finite number of homomorphisms $$\rho_1,\ldots,\rho_n:\mathcal{T}(S)\rightarrow
\PSp(2g-2,\mathbb{Z})$$ with the following properties:
\begin{itemize}
\item[1.] The homomorphism $\rho_i$ is onto a finite index subgroup of
$\PSp({2g-2,\mathbb{Z}})$ for every $1\le i \le n$.
\item[2.] If $\Psi\in\mathcal{T}(S)$ is not pseudo-Anosov then there
exist $1 \le i \le n$ such that the characteristic polynomial of
$\rho_i(\Psi)$ is reducible.
\end{itemize}
\end{prop}
\begin{proof} Let $\Psi\in\mathcal{T}(S)$ be a non pseudo-Anosov element.

First assume that there exist  $\psi\in\Psi$ satisfying
$\psi(a_1)=a_1^{\pm 1}$ or $\psi(c_k)=c_k^{\pm 1}$ for some $1\le k
\le g-2$. In the first case $\rho(\Psi)$ preserves $\langle
z_1\rangle$ while in the second case, it preserves $\langle
z_1,z_{-1},\ldots,z_{k},z_{k}\rangle $ ($\psi$ leaves the two
components of $S\setminus c_k$ invariant since it acts trivially on
the homology). Thus, if we choose an isomorphism
$\delta:\PSp(\mathcal{H})\rightarrow \PSp(2g-2,\mathbb{Z})$ then the
characteristic polynomial of $\delta\circ\rho(\Psi)$ is reducible.

Now for a general $\Psi$, Lemma $\ref{simple criterion}$ shows that
there is an orientation preserving homeomorphism $\nu$ of $S$ such
that $(1)$ or $(2)$ of this lemma are satisfied.  We can use the
above construction of the double cover with respect to
$\nu(a_1),\nu(b_1),\ldots,\nu(a_g),\nu(b_g)$ instead of
${a}_1,{b}_1,\ldots ,{a}_g,{b}_g$ to get the desired homomorphism.

This implies that for every $\Psi\in\mathcal{T}(S)$ which is not
pseudo-Anosov there are a double cover $(\tilde{S},\varphi)$ of $S$,
an orientation preserving homeomorphism $\nu$ of $S$ and a
homomorphism $\rho_{\tilde{S},\nu}:\mathcal{T}(S)\rightarrow
\PSp(2g-2,\mathbb{Z})$ such that the characteristic polynomial of
$\rho_{\tilde{S},\nu}(\Psi)$ is reducible. However, the dependence
of $\rho_{\tilde{S},\nu}$ on $\nu$ just follows from the choice of
basis for $\ker (\Ho_1(\tilde{S},\mathbb{Z})\rightarrow
\Ho_1({S},\mathbb{Z}))$.

Hence, fix once and for all a basis $B_{\tilde{S}}$ of $\ker
(\Ho_1(\tilde{S},\mathbb{Z})\rightarrow \Ho_1({S},\mathbb{Z}))$
($B_{\tilde{S}}$ does not depend on $\nu$ or $\Psi$ only on
$\tilde{S}$) and define the homomorphism
$\rho_{\tilde{S}}:\mathcal{T}(S)\rightarrow \PSp(2g-2,\mathbb{Z})$
with respect to this basis. Then, $\rho_{\tilde{S}}(\Psi)$ and
$\rho_{\tilde{S},\nu}(\Psi)$ are conjugate so the characteristic
polynomial of $\rho_{\tilde{S}}(\Psi)$ is also reducible. This
finish the proof since $S$ only  has $2^{2g}-1$ double covers.
\end{proof}

\section{Random walks and large sieve techniques}

We start this section with a formal model for random walks. Let
$\Gamma$ be a group. A finite symmetric subset $\Sigma$ of $\Gamma$
is called \emph{admissible} if it generates $\Gamma$ and satisfies
an odd relation, e.g. if it contains the identity (symmetric means
$\Sigma=\Sigma^{-1}$). Fix an admissible subset $\Sigma$ of $\Gamma$
and endow it with the uniform probability measure. A \emph{walk} $w$
on $\Gamma$ with respect to $\Sigma$ is a function
$w:\mathbb{N}^+\rightarrow \Sigma$. The $k^{\textrm{th}}$-step of
$w$ is $w_k:=w(1)\cdots w(k)$ (in particular, $w_0$ is the
identity). The probability measure on $\Sigma$ induces a product
probability measure on the set of $\Sigma$-walks $\Sigma^{\N^{+}}$.
For a subset $Z$ of $\Gamma$ we denote the probability that the
$k^{\textrm{th}}$-step of a walk belongs to $Z$ by $\prob(w_k \in
Z)$. The set $Z$ is called \emph{exponentially small with respect to
$\Sigma$} if there are constants $c,\alpha>0$ such that $\prob(w_k
\in Z) \le ce^{-\alpha k}$ for every $k \in \N$. A set is
\emph{exponentially small} if it exponentially small with respect to
every admissible subset of $\Gamma$.

The following theorem follows immediately from Theorem 2 of
\cite{LM}:

\begin{thm}\label{sieve theorem} Let $\Gamma$ be a finitely generated
group and let $\mathcal{P}$ the set of all but finitely many primes.
Let $(N_p)_{p \in \mathcal{P}}$ be a series of finite index normal
subgroups of $\Gamma$. Assume that there is a constant  $d \in \N^+$
such that:
\begin{itemize}
\item[1.] $\Gamma$ has property-$\tau$ with respect to the series of
normal subgroup $(N_p \cap N_q)_{p,q\in \mathcal{P}}$.
\item[2.] $|\Gamma_p| \le p^d$ for every $p \in \mathcal{P}$ where $\Gamma_p:=\Gamma/N_p$.
\item[3.] The natural map $\Gamma_{p,q}\rightarrow \Gamma_p \times
\Gamma_q$ is an isomorphism   for every distinct $p,q\in
\mathcal{P}$ where $\Gamma_{p,q}:=\Gamma /(N_p \cap N_q)$.
\end{itemize} Then a subset $Z\subseteq \Gamma$ is exponentially small if there is $c>0$ such that:
\begin{itemize}
\item[4.] $\frac{|Z_p|}{|\Gamma_p|}\le 1-c$ for every $p \in \mathcal{P}$ where $Z_p:=ZN_p/N_p$.
\end{itemize}
\end{thm}

See Section 2 of \cite{LM} or \cite{Lu} for the definition of
property-$\tau$. We are ready to prove our main theorem:

\begin{thm}\label{main thm} Let $S$ be an orientable connected compact surface of genus $g \ge
3$ and denote its Torelli subgroup by $\mathcal{T}(S)$. The set
$Z\subseteq \mathcal{T}(S)$ consisting of non-pseudo-Anosov elements
is exponentially small.
\end{thm}
\begin{proof} Let $\rho_1,\ldots,\rho_n:\mathcal{T}(S)\rightarrow
\PSp(2g-2,\mathbb{Z})$ be the homomorphisms of Proposition \ref{main
prop}. For every $1 \le i \le n$, define $Z_i:=\rho_i^{-1}(R)$ where
$R \subseteq \PSp(2g-2,\Z)$ consisting of the elements with
reducible characteristic polynomial. We have $Z \subseteq \cup_{1
\le  i \le n}Z_i$ so it is enough to prove that for such an $i$ the
set $Z_i$ is exponentially small. Fix $1 \le i \le n$.

For every prime $p$ define $\psi_p:=\pi_p \circ \rho_i$ and
$N_p:=\ker \psi_p$ where $\pi_p$ is the modulo-$p$ map
$\PSp(2g-2,\Z) \rightarrow \PSp(2g-2,\Z/p\Z)$. The image of $\rho_i$
in $P\Sp(2g-2,\Z)$ is of finite index. This and the fact that
$\PSp(2g-2,\Z/p\Z)$ is simple for $p \ge 3$ implies that for large
enough prime $p$ the image of $\psi_p$ is $\PSp(2g-2,\Z/p\Z)$. Let
$\mathcal{P}$ be the set of large enough primes. We have to check
that the four conditions of Theorem \ref{sieve theorem} are
satisfied.

Condition $1$ follows from the fact that $\PSp(2g-2,\Z)$ has
property-$\tau$ with respect to the family of congruence subgroups
and the fact that property-$\tau$ is inherited by finite index
subgroups (see for example \cite{Lu}). Condition $2$ is readily true
for $d=(2g-2)^2$. Condition $3$ follows from the fact that for
distinct primes $p,q \ge 3$ the groups $\PSp(2g-2,\Z/p\Z)$ and
$\PSp(2g-2,\Z/q\Z)$ are simple and non-isomorphic. Finally condition
$4$ was proved by Chavdarov \cite{Ch} (see also \cite{Ri}).
\end{proof}

{\flushleft{\bf Remark.} The method of proof actually applies to
many other subgroups of $\MCG(S)$ besides the Torelli subgroup. In
fact, by combining strong approximation \cite{We} and the result of
Salehi-Golsefidy-Varju about property-$\tau$ \cite{SGV}, the proof
applies to all subgroups $\Gamma$ of $\mathcal{T}(S)$ such that for
every $1 \le i \le n$ the group $\rho_i(\Gamma)$ is Zariski-dense
subgroup of $\PSp(2g-2,\mathbb{Z})$ where the notations are as in
Proposition \ref{main prop}. In particular, the result applies to
all finite index subgroup of $\mathcal{T}(S)$.}

Another interesting class of subgroups  are the Johnson subgroups.
For every $k \in \N^+$ let $J(k)$ be the kernel of the action of
$\MCG(S)$ on $\pi_1(S)/\gamma_i(\pi_1(S))$ where $\pi_1(S)$ is the
fundamental group of $S$ and $\gamma_i(\pi_1(S))$ is the
$i^{\text{th}}$-subgroup in the lower central series of $\pi_1(S)$.
The Torelli group is just $J(2)$ and for every $k \ge 2$ the group
$J(2)/J(k)$ is nilpotent. The group $\PSp(2g-2,\Z)$ has property-$T$
for $g \ge 3$ so if $H$ is a finite index subgroup of
$\PSp(2g-2,\Z)$ and $L \lhd H$ is co-nilpotent then $L$ is of finite
index in $H$. Thus, for every $k \ge 2$ and every $1 \le i \le n$
the group $\rho_i(J(k))$ is a finite index subgroup of
$\PSp(2g-2,\mathbb{Z})$ and in particular Zariski-dense.

For $k \ge 3$ the group $J(k)$ is unlikely to be finitely generated.
However the previous discussion shows that there is a finite set
$R(k) \subseteq J(k)$ such that if $L(k)$ is a finitely generated
subgroup of $J(k)$ which contains $R(k)$ then for every $1 \le i \le
n$ the group $\rho_i(L(k))$ is a Zariski-dense subgroup of
$\PSp(2g-2,\mathbb{Z})$. Thus, the set of non-pseudo-Anosov elements
of $L(k)$ is exponentially small.

This shows that in some sense Theorem \ref{main thm} is true for
"sufficiently large" finitely generated subgroups of
$\mathcal{T}(S)$, e.g. subgroups which contains $R(k)$.  The use
here of the term "sufficiently large" should be compared to the one
in \cite{Ma} where the term "sufficiently large" subgroup refer to a
subgroup which contains a pair of pseudo-Anosov elements with
distinct fixed points in the space of projective measured
laminations.

 %In fact, the above theorem remains true for finitely
%generated subgroups of the Torelli group such that their image under
%the homomorphism $\rho$ of Proposition \ref{rho} is a finite index
%subgroup of $\PSp(\mathcal{H})$. {\bf \Large Alex. you wanted to add
%something.}

\end{document}